\documentclass[letterpaper, 10 pt, conference]{ieeeconf}  

\IEEEoverridecommandlockouts                              
\usepackage[utf8]{inputenc}
\usepackage[T1]{fontenc}

\usepackage{amsmath,amssymb,amsfonts}

\usepackage{algorithmic}
\usepackage{graphicx}
\usepackage{array}
\usepackage{booktabs}
\usepackage{siunitx}

\usepackage[table]{xcolor}

\usepackage{tikz}
\usepackage{pgfplots}
\usepackage{tikz-3dplot}

\pgfplotsset{compat=1.18}
\usepgfplotslibrary{colormaps}
\usetikzlibrary{arrows.meta,calc,positioning}

\usepackage[font=small,labelfont=bf]{caption}
\usepackage{subcaption}
\usepackage{cite}

\usepackage{pifont}

\newcommand{\cmark}{%
    \textcolor{green!60!black}{\ding{51}}%
}

\newcommand{\xmark}{%
    \textcolor{red}{\ding{55}}%
}

\definecolor{headergray}{gray}{0.9}
\definecolor{ssmagreen}{RGB}{220,245,220}
\definecolor{SFRrow}{HTML}{E67E22}
\definecolor{FRrow}{HTML}{2E86C1}

\newtheorem{lemma}{\textbf{Lemma}}

\newtheorem{definition}{\textbf{Definition}}
\newtheorem{proposition}{\textbf{Proposition}}

\usepackage[
    colorlinks=true,
    linkcolor=purple,
    citecolor=blue,
    urlcolor=green
]{hyperref}

\AtBeginDocument{%
    \hypersetup{
    colorlinks=true,
    linkcolor=purple,
    citecolor=blue,
    urlcolor=green
    }%
}

\title{\LARGE \bf
A Fast and High-Accuracy Finite-Time Sliding-Mode Algorithm for Computing Local Stackelberg Equilibria in Nonlinear Bilevel Games*
}

\author{Abdelhafid Zenati$^{1,2}$,~ Dave D Muir$^{1}$,~Kamal Youcef-Toumi$^{2}$
\thanks{*This work was not supported by any organization}
\thanks{$^{1}$ Engineering Department, School of Science and Technology (SST), City St George's University of London, Northampton Square,  
		London, EC1V 0HB, United Kingdom.
        {\tt\small  (e-mail: \{Abdelhafid.Zenati, Dave.Muir\}@city.ac.uk).}}%
\thanks{$^{2}$ Department of Mechanical Engineering, Massachusetts Insti
tute of Technology, Cambridge, Massachusetts 02139, United
States. {\tt\small (e-mail: \{youcef, abdelzen\}@mit.edu)}
        }%
}

\begin{document}

\maketitle
\thispagestyle{empty}
\pagestyle{empty}

\begin{abstract}
Computing local Stackelberg equilibria in nonlinear sequential games is challenging because gradient-based best-response methods may diverge, cycle, or converge to stationary points that do not satisfy the required equilibrium conditions. These failures can result from curvature asymmetry, unstable response mappings, and lack of invariance with respect to the follower's best-response manifold. This paper introduces the Stackelberg Sliding-Mode Algorithm (SSMA), a control-theoretic framework that recasts Stackelberg equilibrium computation as a manifold-stabilization problem. The follower and leader first-order optimality conditions define Stackelberg sliding manifolds, and sliding-mode dynamics drive the iterates to these manifolds in finite time. The follower curvature and leader reduced-curvature conditions determine manifold attractivity, making valid local Stackelberg equilibria attractive while rejecting stationary points that fail the second-order equilibrium conditions. The resulting approach integrates finite-time manifold stabilization, Stackelberg-consistent dynamics,
and curvature-based equilibrium selection within a unified framework. Across 500 nonconvex Monte Carlo trials, SSMA achieved 100\% success, reached the prescribed tolerance in a median of 26 iterations, compared with 456 for the nearest fully successful baselines, and attained terminal residuals near \(10^{-16}\), while converging to a valid local Stackelberg equilibrium.
\end{abstract}

\section{INTRODUCTION}

Stackelberg games provide a foundational framework for modelling hierarchical decision-making in economics, security, biology,  energy systems, and engineering \cite{cdc1,cdc19,cdc18}. The defining feature is the leader’s ability to anticipate the follower’s reaction and optimise accordingly. This strategic asymmetry has been widely exploited in applications ranging from security and allocation problems to cyber-physical systems and energy markets \cite{cdc2,cdc3}. Recent studies highlight the importance of accurate equilibrium computation in these settings, particularly when the interaction between decision makers involves nonconvex objectives, implicit constraints, and potentially unstable responses \cite{cdc4, cdc5}.

Despite this breadth of application, computing Stackelberg equilibria remains difficult, especially in nonlinear or nonconvex settings \cite{cdc5,cdc17}. Classical reformulation approaches treat the follower’s optimality conditions as complementarity constraints embedded in the leader’s problem. These methods can be effective in structured environments, such as linear–quadratic or convex models, but they often rely on strong assumptions: uniqueness of the follower’s equilibrium, specific equilibrium selection rules, or restrictive curvature conditions. When these assumptions fail, reformulation-based algorithms can converge to poor solutions or become computationally intractable, especially when the feasible reaction set is complex or multi-valued \cite{cdc6}. Further, the computational hardness of robust or constrained variants of Stackelberg equilibria has been demonstrated in several recent studies, revealing the sensitivity of leader strategies to small deviations in follower behaviour \cite{cdc7}.

The limitations of gradient-based best-response dynamics have also been
extensively documented: these methods may diverge, produce oscillations,
or converge to stationary points that do not correspond to valid
Stackelberg equilibria \cite{cdc8}. The underlying instability is linked
to curvature asymmetry, ill-conditioned response mappings, and the
general lack of invariance of gradient iterates to the true best-response
manifold \cite{cdc9}. First-order schemes, including gradient-based,
optimistic-gradient (OGDA) and nested best-response methods (BR-GD), seek to improve the underlying descent dynamics but may remain sensitive to the geometry and conditioning of the follower response \cite{liu2022bome}. More structured approaches, such as Follow-the-Ridge (FR), incorporate the local response geometry to track the follower solution manifold, while curvature-aware variants such as HessianFR further exploit second-order information \cite{bao2023finding}. Attempts to
circumvent these issues using differentiation through equilibria and
implicit-function-theorem (IFT) methods provide direct estimates of the leader gradient through the follower response, but can require solving or regularizing Hessian systems and may become sensitive to poor conditioning \cite{alesiani2023implicit}. Higher-order approaches, including  cubic-regularized Newton schemes (CRN), further exploit curvature
to improve local equilibrium computation \cite{bi2026second}. More generally,
differentiation through equilibria, implicit-function methods, and
end-to-end learning offer greater flexibility but still depend on
uniqueness assumptions or high-quality equilibrium-selection oracles,
which are rarely available in general nonconvex settings
\cite{cdc10, cdc11}.

Addressing these challenges requires a framework that directly stabilises
the conditions defining a local Stackelberg equilibrium. Motivated by this
need, we propose a control-theoretic perspective grounded in sliding-mode
concepts \cite{cdc12}. A local Stackelberg equilibrium is characterised by
two first-order conditions, Follower Stationarity and Leader Stationarity,
together with two second-order conditions, Follower Curvature and Leader
Reduced Curvature, see Definition \ref{def1}. Here, Leader Stationarity is defined by the total
derivative of the reduced leader objective along the follower's response,
rather than by the partial gradient with respect to the leader variable.
We formulate the two stationarity conditions as Stackelberg sliding
manifolds, whose joint zero set defines the first-order Stackelberg critical
set. The update law is then designed so that these manifolds become
attractive only when the corresponding Follower Curvature and Leader Reduced
Curvature conditions hold. In this way, the dynamics are driven towards
points satisfying both the first- and second-order requirements of a local
Stackelberg equilibrium, while non-equilibrium stationary points are
excluded. This sliding-mode formulation provides fast finite-time reaching
of the Stackelberg manifolds together with robust invariance once they are
reached. Numerical results on representative bilevel problems demonstrate that SSMA does not merely reach a stationary point; it rapidly and reliably converges to a valid local Stackelberg equilibrium with near-machine-precision accuracy.

 The remainder of the paper is organised as follows. Section~\ref{sec2} introduces the general-sum Stackelberg game formulation.  The proposed SSMA (SSMA) presented in Section~\ref{sec3}. Section~\ref{sec4} provides its convergence analysis. Section~\ref{sec5} reports validation and numerical experiments. The paper is concluded in Section~\ref{sec6}.
\section{Problem Setting}\label{sec2}

We consider a general-sum Stackelberg game involving two strategic agents: a leader with decision variable \( x \in \mathbb{R}^n \) and a follower with decision variable \( y \in \mathbb{R}^m \). 
The leader chooses an action first, after which the follower observes \( x \) and selects \( y \) in response.
The leader aims to minimise a cost function \( f(x,y) \) while the follower minimises \( g(x,y) \).
This hierarchical interaction induces the bilevel optimisation problem (see \cite{cdc20,cdc21} for instance)
\[
\min_{x \in \mathbb{R}^n} \left\{ f(x,y) \;\middle|\; y \in \arg\min_{y \in \mathbb{R}^m} g(x,y) \right\}.
\]
The follower’s optimal reaction set is denoted by
\[
R(x) := \arg\min_{y} g(x,y).
\]
\begin{definition}[\color{blue} Stackelberg Equilibrium]
A pair \( (x^*, y^*) \) is a global Stackelberg equilibrium if
\[
y^* \in R(x^*),~ \text{and} ~
f(x^*, y^*) \le \max_{y \in R(x)} f(x,y), \;\; \forall x \in \mathbb{R}^n.
\]
\end{definition}

Because global equilibria are difficult to obtain in nonlinear settings, we also adopt a local notion based on neighbourhood behaviour and the implicit follower response.

\begin{definition}[\color{blue}Local Stackelberg Equilibrium] \label{def1}
A pair \( (x^*, y^*) \) is a local Stackelberg equilibrium if:
\begin{enumerate}
    \item \( y^* \) is a local minimiser of \( g(x^*, \cdot) \);
    \item There exists an implicit function \( r(x) \) satisfying 
    \( \nabla_y g(x, r(x)) = 0 \) locally with \( r(x^*) = y^* \),
    and \( x^* \) is a local minimiser of the reduced function
    \[
    \phi(x) := f\big(x, r(x)\big).
    \]
\end{enumerate}
\end{definition}
Importantly, local Stackelberg stationarity is not joint stationarity.
Since the leader anticipates the follower's response \(r(x)\), Leader
Stationarity is defined by the total derivative of the reduced objective
\(\phi(x)=f(x,r(x))\), rather than by \(\nabla_x f(x,y)\) alone.
Thus, the first-order conditions are Follower Stationarity
\(\nabla_y g=0\) and Leader Stationarity \(\nabla\phi=0\).

We use the standard block-Hessian notation:
\[
\nabla^2 f(x,y)=
\begin{bmatrix}
H_{xx} & H_{xy}\\
H_{yx} & H_{yy}
\end{bmatrix},
~
\nabla^2 g(x,y)=
\begin{bmatrix}
G_{xx} & G_{xy}\\
G_{yx} & G_{yy}
\end{bmatrix}.
\]

\begin{proposition}[\color{blue}Necessary Conditions] \label{cond1}
If \( (x^*, y^*) \) is a local Stackelberg equilibrium, then:
\begin{enumerate}
    \item Follower stationarity:
        \(
    \nabla_y g(x^*, y^*) = 0;
    \)
    \item Leader stationarity on the reaction manifold:
    \[
    \nabla_x f(x^*,y^*) 
    - G_{xy}G_{yy}^{-1}\nabla_y f(x^*,y^*) = 0;
    \]
    \item Follower curvature:
    \(
    \nabla_{yy}^2 g(x^*,y^*) \succeq 0;
    \)
    \item Leader reduced curvature:
\(H_{xx} - H_{xy}G_{yy}^{-1}G_{yx}
    - \nabla_x (G_{xy}G_{yy}^{-1} \nabla_y f)
    + \nabla_y (G_{xy}G_{yy}^{-1}\nabla_y f)\, G_{yy}^{-1}G_{yx}
 \succeq 0.\)
\end{enumerate}
\end{proposition}

\begin{proposition}[\color{blue}Sufficient Conditions] \label{cond2}
A point \( (x,y) \) is a local Stackelberg equilibrium if:
\begin{enumerate}
    \item \( \nabla_y g(x,y) = 0; \)
    \item \( \nabla_x f(x,y) - G_{xy}G_{yy}^{-1}\nabla_y f(x,y) = 0; \)
    \item \( \nabla_{yy}^2 g(x,y) \succ 0; \)
    \item \(H_{xx} - H_{xy}G_{yy}^{-1}G_{yx}
    - \nabla_x (G_{xy}G_{yy}^{-1} \nabla_y f)\\
    + \nabla_y (G_{xy}G_{yy}^{-1}\nabla_y f)\, G_{yy}^{-1}G_{yx}
    \succ 0.\)
    
\end{enumerate}
\end{proposition}

The first-order conditions identify the critical set of possible Stackelberg candidates, whereas the second-order conditions specify which points are locally stable. 
Since computing the reaction function \( r(x) \) explicitly is typically intractable, practical algorithms must enforce these conditions indirectly, motivating the control-theoretic framework developed in subsequent sections.

\section{Stackelberg Sliding-Mode Algorithm (SSMA)}\label{sec3}

Computing local Stackelberg equilibria in nonlinear sequential games requires 
stabilising the coupled leader--follower dynamics so that the iterates evolve 
consistently along the equilibrium manifold. The Key idea here is inspired by sliding--mode control,  where we treat the first--order optimality conditions in Proposition \ref{cond1} and Proposition \ref{cond2} of both agents as sliding 
manifolds (termed \textit{Stackelberg Sliding Manifolds}) that must be reached and preserved if and only if the second--order optimality conditions are satisfied to guarantee local equilibrium behaviour. The SSMA develops a stabilised dynamical system that drives the trajectories toward these manifolds in finite time while maintaining invariance on the follower's implicit response surface. To this end, let the follower's Stackelberg sliding surface  defined as
 \begin{equation}
    \zeta(x,y) = \nabla_y g(x,y),
\end{equation}
which vanishes precisely when the follower satisfies its local optimality condition. Also, let the leader's Stackelberg sliding surface given by
\begin{equation}
\xi(x,y) 
= \nabla_x f(x,y)
- G_{xy}\, G_{yy}^{-1}\nabla_y f(x,y),
\end{equation}
and encodes the leader's stationarity condition restricted to the reaction 
manifold determined by the follower’s best response. Since a local Stackelberg equilibrium must satisfy both \( \zeta(x,y)=0 \) and \( \xi(x,y)=0 \), the SSMA framework constructs update laws that force the system toward this joint manifold and ensure that the dynamics remain invariant once sliding is established. The update laws are given by:
\begin{align}
x_{t+1} &= x_t - \eta_x \xi_t   + \boldsymbol{\sigma}_x(x_t,y_t),\label{updatex} \\
y_{t+1} &= y_t - \eta_y \zeta_t + \boldsymbol{\sigma}_y(x_t,y_t),\label{updatey}
\end{align}

where $\boldsymbol{\sigma}_x$ and $\boldsymbol{\sigma}_y$ are the correction terms given by
\begin{equation}
    \begin{bmatrix}
\boldsymbol{\sigma}_x (x_t,y_t) \\
\boldsymbol{\sigma}_y(x_t,y_t) 
\end{bmatrix}
=
 \mathcal{H}_t^{-1}
\begin{bmatrix}
\eta_x G_{yx}(x_t,y_t) \xi_t - \epsilon_\zeta \operatorname{sgn}(\zeta_t) \\
\eta_y F_{xy}(x_t,y_t) \zeta_t - \epsilon_\xi \operatorname{sgn}(\xi_t)
\end{bmatrix},
\end{equation}
and 
\begin{equation}\label{matrices}
    \mathcal{H}_t 
=\begin{bmatrix}
G_{yx}(x_t,y_t) & G_{yy}(x_t,y_t) \\
\tilde{H}_x(x_t,y_t) & F_{xy}(x_t,y_t)
\end{bmatrix},~\begin{bmatrix}\xi_t \\\zeta_t\end{bmatrix}=\begin{bmatrix}\xi(x_t,y_t)\\\zeta(x_t,y_t)\end{bmatrix},
\end{equation}
and the joint Stackelberg Hessian $\mathcal{H}_t $ of follower and leader are given by
\begin{equation}\label{blocks}
    G_{yx},\, G_{yy} = \nabla \zeta(x_t,y_t), ~
\tilde{H}_x,\, F_{xy} =\nabla \xi(x_t,y_t)
\end{equation}
\begin{table}[!h]
\centering
\caption{Stackelberg Sliding-Mode Algorithm.}
\begin{tabular}{l}
\hline
\textbf{Algorithm I: Stackelberg Sliding Mode Algorithm (SSMA) } \\
\hline \\[-0.7em]
\textbf{Require:} Stepsizes $\eta_x, \eta_y$; switching gains $\epsilon_\zeta, \epsilon_\xi$; iterations $T$. \\[0.2em]
1:\textbf{For} $t = 1,\ldots, T$ \textbf{do} \\[0.2em]
2:\quad The Sliding Manifolds \\
$ \quad ~~\begin{matrix}
    \zeta_t & \leftarrow &\nabla_y g(x_t, y_t) \quad\quad\quad\quad\quad\quad\\
    \xi_t &\leftarrow &\nabla_x f(x_t,y_t) - G_{xy}G_{yy}^{-1} \nabla_y f(x_t,y_t)
\end{matrix}$\\[0.7em]
3:\quad Compute Jacobian blocks of follower and leader: \\\\[-0.4em]~~\quad 
\(
G_{yx},\, G_{yy} = \nabla^2 g(x_t,y_t), ~
\tilde{H}_x,\, F_{xy} =\nabla \xi(x_t,y_t)
\)

\\[-0.5em]
\\
4:\quad Compute the joint Stackelberg Hessian $\mathcal{H}_t $ and its inverse : \\
~~\quad $\mathcal{H}_t 
=\begin{bmatrix}
G_{yx}(x_t,y_t) & G_{yy}(x_t,y_t) \\
\tilde{H}_x(x_t,y_t) & F_{xy}(x_t,y_t)
\end{bmatrix}$

\\[-0.5em]
\\
5:\quad Compute the corrector terms\\~~\quad \textcolor{blue}{$
\begin{bmatrix}
\boldsymbol{\sigma}_x \\[2pt] \boldsymbol{\sigma}_y
\end{bmatrix}
\leftarrow 
\mathcal{H}_t^{-1}
\begin{bmatrix}
\eta_x \, G_{yx}\, \xi_t - \epsilon_\zeta\,\mathrm{sgn}(\zeta_t) \\
\eta_y \, F_{xy}\, \zeta_t - \epsilon_\xi\,\mathrm{sgn}(\xi_t)
\end{bmatrix}$} \\[0.6em]

6:\quad $x_{t+1} \leftarrow x_t - \eta_x \, \xi_t + \textcolor{blue}{\boldsymbol{\sigma}_x}$ \\[0.4em]

7:\quad $y_{t+1} \leftarrow y_t - \eta_y \, \zeta_t \;+\; \textcolor{blue}{\boldsymbol{\sigma}_y}$ \\[1em]

\hline
\label{table:SFRSG}
\end{tabular}
\end{table}

Overall, the SSMA method combines sliding-mode control with bilevel geometry to deliver a robust, stabilised update law for computing local Stackelberg equilibria in general nonlinear settings, avoiding the oscillations, divergence, and degeneracies that affect classical gradient methods.

\section{SSMA Convergence Analysis}\label{sec4}
In the following lemma, we establish the sliding reaching law satisfied by the 
proposed SSMA. This result characterises how the 
Stackelberg sliding manifolds  evolve under the SSMA updates and provides 
the basis for the subsequent convergence analysis.

\begin{lemma}[\color{blue} SSMA Sliding Reaching Law]

Let \(f(x,y)\) and \(g(x,y)\) be twice continuously differentiable functions and  define the sliding manifolds:
\[
\zeta_t = \nabla_y g(x_t,y_t),~
\xi_t = \nabla_x f(x_t,y_t) - G_{xy}G_{yy}^{-1}\nabla_y f(x_t,y_t).
\]

Assume the iterates \((x_t,y_t)\) are updated according to \eqref{updatex}-\eqref{updatey},
Then the updated Stackelberg sliding Manifolds satisfy the recurrence
\begin{equation}
\begin{bmatrix}
\zeta_{t+1}\\
\xi_{t+1}
\end{bmatrix}
=
\mathcal{A}
\begin{bmatrix}
\zeta_t\\
\xi_t
\end{bmatrix}
-
\begin{bmatrix}
\epsilon_\zeta \operatorname{sgn}(\zeta_t)\\
\epsilon_\xi \operatorname{sgn}(\xi_t)
\end{bmatrix}
+
\begin{bmatrix}
\boldsymbol{\nu}_t\\
\boldsymbol{\mu}_t
\end{bmatrix},
\end{equation}
where \( \boldsymbol{\nu}_t,\boldsymbol{\mu}_t \) denote second-order Taylor 
remainder terms arising from the gradient expansion and
\begin{equation}
\mathcal{A}
=
\begin{bmatrix}
I - \eta_y G_{yy} & 0 \\
0 & I - \eta_x \tilde H_x
\end{bmatrix}.
\end{equation}
Also, the matrix blocks $G_{yy}$ and $\tilde H_x$ are given in \eqref{blocks}.\label{lemma_stack}
\end{lemma}
\begin{proof}
We begin by expanding the sliding variables at the next iteration using Taylor's theorem around \( (x_t,y_t) \):
\begin{align}
\zeta_{t+1} 
&=  \zeta_t + G_{yx}\,\delta x + G_{yy}\,\delta y + \boldsymbol{\nu}_t, 
\label{eq:zeta-taylor} \\
\xi_{t+1} 
&
= \xi_t + \tilde{H}_x\,\delta x + F_{xy}\,\delta y + \boldsymbol{\mu}_t,
\label{eq:xi-taylor}
\end{align}
where
\[
\delta x = x_{t+1}-x_t, 
\quad 
\delta y = y_{t+1}-y_t,
\]
and the remainder terms satisfy
\[
\|\boldsymbol{\nu}_t\|,\;\|\boldsymbol{\mu}_t\|
= \mathcal{O}\!\left(\|(\delta x,\delta y)\|^2\right) 
\]
and 
\begin{equation}
   \lim\limits_{(\delta x,\delta y)\rightarrow (0,0)}\boldsymbol{\nu}_t(\delta x,\delta y)=0,~\lim\limits_{(\delta x,\delta y)\rightarrow (0,0)}\boldsymbol{\mu}_t(\delta x,\delta y)=0.
\end{equation}
From Algorithm~\ref{table:SFRSG}, the update law is
\begin{align}
\delta x = -\eta_x \,\xi_t + \boldsymbol{\sigma}_x, 
~
\delta y = -\eta_y \,\zeta_t + \boldsymbol{\sigma}_y.
\label{eq:deltay-stack}
\end{align}

Substituting \eqref{eq:deltay-stack} into \eqref{eq:zeta-taylor} gives
\begin{align}
\zeta_{t+1}
&= \zeta_t + G_{yx}(-\eta_x \xi_t + \boldsymbol{\sigma}_x)
   + G_{yy}(-\eta_y \zeta_t + \boldsymbol{\sigma}_y)
   + \boldsymbol{\nu}_t \notag\\
&= (I-\eta_y G_{yy})\zeta_t
   - \eta_x G_{yx}\xi_t
   + G_{yx}\boldsymbol{\sigma}_x
   + G_{yy}\boldsymbol{\sigma}_y
   + \boldsymbol{\nu}_t.
\label{eq:zeta-coupled}
\end{align}

Similarly, substituting \eqref{eq:deltay-stack} into \eqref{eq:xi-taylor} yields
\begin{align}
\xi_{t+1}
&= \xi_t + \tilde{H}_x(-\eta_x \xi_t + \boldsymbol{\sigma}_x)
   + F_{xy}(-\eta_y \zeta_t + \boldsymbol{\sigma}_y)
   + \boldsymbol{\mu}_t \notag\\
&= (I-\eta_x \tilde{H}_x)\xi_t
   - \eta_y F_{xy}\zeta_t
   + \tilde{H}_x\boldsymbol{\sigma}_x
   + F_{xy}\boldsymbol{\sigma}_y
   + \boldsymbol{\mu}_t.
\label{eq:xi-coupled}
\end{align}

Now choose the sliding correction as
\begin{equation}
\begin{bmatrix}
\boldsymbol{\sigma}_x\\[2pt]
\boldsymbol{\sigma}_y
\end{bmatrix}
=
\mathcal{H}_t^{-1}
\begin{bmatrix}
\eta_x G_{yx}\xi_t - \epsilon_\zeta\,\mathrm{sgn}(\zeta_t)\\[4pt]
\eta_y F_{xy}\zeta_t - \epsilon_\xi\,\mathrm{sgn}(\xi_t)
\end{bmatrix},
\label{eq:sigma-stack}
\end{equation}
where $\mathcal{H}_t $ is define in \eqref{matrices}. By construction of \eqref{eq:sigma-stack}, its first row gives
\[
G_{yx}\boldsymbol{\sigma}_x + G_{yy}\boldsymbol{\sigma}_y
= \eta_x G_{yx}\xi_t - \epsilon_\zeta\,\mathrm{sgn}(\zeta_t),
\]
which exactly cancels the mixed coupling term
\(
-\eta_x G_{yx}\xi_t
\)
in \eqref{eq:zeta-coupled}. Hence,
\begin{equation}
\zeta_{t+1}
=
(I-\eta_y G_{yy})\zeta_t
-\epsilon_\zeta\,\mathrm{sgn}(\zeta_t)
+\boldsymbol{\nu}_t.
\label{eq:zeta-reaching}
\end{equation}

Likewise, the second row of \eqref{eq:sigma-stack} yields
\[
\tilde{H}_x\boldsymbol{\sigma}_x + F_{xy}\boldsymbol{\sigma}_y
= \eta_y F_{xy}\zeta_t - \epsilon_\xi\,\mathrm{sgn}(\xi_t),
\]
which cancels the mixed coupling term
\(
-\eta_y F_{xy}\zeta_t
\)
in \eqref{eq:xi-coupled}. Therefore,
\begin{equation}
\xi_{t+1}
=(I-\eta_x \tilde{H}_x)\xi_t
-\epsilon_\xi\,\mathrm{sgn}(\xi_t)
+\boldsymbol{\mu}_t.
\label{eq:xi-reaching}
\end{equation}

Stacking \eqref{eq:zeta-reaching} and \eqref{eq:xi-reaching}, we obtain the compact recurrence
\begin{equation}
\begin{bmatrix}
\zeta_{t+1}\\[4pt]
\xi_{t+1}
\end{bmatrix}
=
\underbrace{
\begin{bmatrix}
I-\eta_y G_{yy} & 0\\[2pt]
0 & I-\eta_x \tilde{H}_x
\end{bmatrix}
}_{\displaystyle \mathcal{A}}
\begin{bmatrix}
\zeta_t\\[4pt]
\xi_t
\end{bmatrix}
-
\begin{bmatrix}
\epsilon_\zeta\,\mathrm{sgn}(\zeta_t)\\[4pt]
\epsilon_\xi\,\mathrm{sgn}(\xi_t)
\end{bmatrix}
+
\begin{bmatrix}
\boldsymbol{\nu}_t\\[4pt]
\boldsymbol{\mu}_t
\end{bmatrix}.
\label{eq:stack-reaching-final}
\end{equation}

\noindent
This shows that, under the proposed sliding correction, the evolution of the 
Stackelberg sliding variables follows a decoupled first--order linear recurrence  with a discontinuous reaching term and higher--order perturbations \cite{cdc22}. Consequently, 
the dynamics implement a discrete--time sliding mode on the manifold
\[
\zeta_t = 0, 
\quad 
\xi_t = 0,
\]
which corresponds to the first--order optimality conditions of a local 
Stackelberg equilibrium. Convergence of the dynamics toward this manifold requires the stability of the 
linear part of the recurrence. In particular, the spectral radius of the matrix 
\(\mathcal{A}\) must satisfy
\[
\rho(\mathcal{A}) < 1.
\]
This condition ensures that the homogeneous component of the system is 
contractive and that the sliding variables asymptotically approach the manifold. 
Since
\[
\mathcal{A} =
\begin{bmatrix}
I - \eta_y G_{yy} & 0 \\
0 & I - \eta_x \tilde{H}_x
\end{bmatrix},
\]
the requirement \(\rho(\mathcal{A}) < 1\) implies that the 
matrices \(G_{yy}\) and \(\tilde{H}_x\) are positive definite. These conditions 
coincide precisely with the second--order optimality conditions of a local 
Stackelberg equilibrium in Definition \ref{def1}, namely that the follower objective is locally convex 
in \(y\) and that the reduced Hessian of the leader objective along the reaction 
manifold is positive definite.
\end{proof}

In classical sliding--mode implementations, the discontinuous 
$\mathrm{sgn}(\cdot)$ function often induces high--frequency oscillations, 
commonly referred to as \emph{chattering}, due to rapid switching in the 
vicinity of the sliding manifold. To alleviate this undesirable effect, the 
discontinuous $\mathrm{sgn}(\cdot)$ term is replaced by a continuous saturation 
function $\mathrm{sat}(\cdot)$ defined within a thin boundary layer \cite{cdc12}. 
This regularisation smooths the control action, improves numerical behaviour, 
and preserves the practical reaching properties of the sliding dynamics. Accordingly, the updated sliding variables satisfy the recurrence
\begin{equation}
\begin{bmatrix}
\zeta_{t+1}\\[4pt]
\xi_{t+1}
\end{bmatrix}
=
\mathcal{A}
\begin{bmatrix}
\zeta_t\\[4pt]
\xi_t
\end{bmatrix}
-
\begin{bmatrix}
\epsilon_\zeta\,\mathrm{sat}(\zeta_t)\\[4pt]
\epsilon_\xi\,\mathrm{sat}(\xi_t)
\end{bmatrix}
+
\begin{bmatrix}
\boldsymbol{\nu}_t\\[4pt]
\boldsymbol{\mu}_t
\end{bmatrix}.
\label{eq:stack-reaching-final}
\end{equation}

On the other hand, the matrices $G_{yy}$ and $\tilde{H}_x$ are symmetric, which implies that the Jacobian associated with~\eqref{eq:stack-reaching-final} has real eigenvalues. Consequently, the local dynamics do not exhibit oscillatory behaviour, and no limit cycles arise in the vicinity of the equilibrium. This ensures a monotonic approach toward the equilibrium manifold without rotational dynamics.

Building on the discrete sliding dynamics in Lemma~\ref{lemma_stack}, we now
analyse the convergence of the proposed SSMA algorithm. The next proposition
establishes conditions for finite-time reaching of the Stackelberg sliding
manifold, defined by $\zeta_t=0$ and $\xi_t=0$. The manifold is attractive
when the required second-order curvature conditions hold. We emphasise that
finite-time reaching concerns the sliding manifold, not finite-time convergence
of $(x_t,y_t)$ to the equilibrium; thereafter, the dynamics converge locally
towards the Stackelberg equilibrium.

\vspace{0.5em}

\begin{proposition}[\color{blue}SSMA  Convergence Analysis]
Let \(f\) and \(g\) be twice continuously differentiable functions and  define the sliding manifolds:
\[
\zeta_t = \nabla_y g(x_t,y_t),~
\xi_t = \nabla_x f(x_t,y_t) - G_{xy}G_{yy}^{-1}\nabla_y f(x_t,y_t).
\]
and assume
\[
\ell_{\zeta} I \preceq G_{yy} \preceq L_{\zeta} I, \quad
\ell_{\xi} I \preceq \tilde H_x \preceq L_{\xi} I,
\]
for some $0<\ell_{\xi}\le L_{\xi}$ and $0<\ell_{\zeta}\le L_{\zeta}$.  

Consider the reaching law given in Lemma~\ref{lemma_stack} with bounded 
perturbations $\|\nu_t\|\le\bar\nu$ and $\|\mu_t\|\le\bar\mu$. Define
\begin{equation}
\begin{array}{ll}
\alpha_\zeta:=\max\{|1-\eta_y\ell_{\zeta}|,|1-\eta_yL_{\zeta}|\},\\
\alpha_\xi:=\max\{|1-\eta_x\ell_{\xi}|,|1-\eta_xL_{\xi}|\},
\end{array}
\end{equation}
and suppose
\[
0<\eta_y<\tfrac{2}{L_{\zeta}},\quad
0<\eta_x<\tfrac{2}{L_{\xi}},
\]
and 
\[
\epsilon_\zeta>\alpha_\zeta\rho_\zeta+\bar\nu,\quad
\epsilon_\xi>\alpha_\xi\rho_\xi+\bar\mu .
\]
Define the in-layer contraction factors
\begin{equation}
\begin{array}{ll}
L_\zeta^{\mathrm{in}}
:=\Big\|I-\eta_y G_{yy}-\tfrac{\epsilon_\zeta}{\rho_\zeta}I\Big\|,
~
L_\xi^{\mathrm{in}}
:=\Big\|I-\eta_x \tilde H_x-\tfrac{\epsilon_\xi}{\rho_\xi}I\Big\|,
\end{array}
\end{equation}
and assume $L_\zeta^{\mathrm{in}}<1$ and $L_\xi^{\mathrm{in}}<1$ along the trajectory.  

Then, with
\[
T_\star:=\max\!\left\{
\Big\lceil\tfrac{\|\zeta_0\|-\rho_\zeta}{\epsilon_\zeta-\alpha_\zeta\rho_\zeta-\bar\nu}\Big\rceil_+,
\;
\Big\lceil\tfrac{\|\xi_0\|-\rho_\xi}{\epsilon_\xi-\alpha_\xi\rho_\xi-\bar\mu}\Big\rceil_+
\right\},
\]
the following conclusions hold:
\begin{enumerate}
\item[\emph{(i)}] For all $t\ge T_\star$, $\|\zeta_t\|\le\rho_\zeta$ and $\|\xi_t\|\le\rho_\xi$.
\item[\emph{(ii)}] For all $t\ge T_\star$,
\begin{equation}
\begin{array}{ll}
\|\zeta_t\|
&\le (L_\zeta^{\mathrm{in}})^{t-T_\star}\|\zeta_{T_\star}\|
+\dfrac{\bar\nu}{1-L_\zeta^{\mathrm{in}}},
\\[6pt]
\|\xi_t\|&
\le (L_\xi^{\mathrm{in}})^{t-T_\star}\|\xi_{T_\star}\|
+\dfrac{\bar\mu}{1-L_\xi^{\mathrm{in}}}.
\end{array}
\end{equation}
\end{enumerate}
Consequently,
\begin{equation}
\begin{array}{ll}
\limsup\limits_{t\to\infty}\|\zeta_t\|
\le \max\!\Big\{\rho_\zeta,\tfrac{\bar\nu}{1-L_\zeta^{\mathrm{in}}}\Big\},
\\[6pt]
\limsup\limits_{t\to\infty}\|\xi_t\|
\le \max\!\Big\{\rho_\xi,\tfrac{\bar\mu}{1-L_\xi^{\mathrm{in}}}\Big\}.
\end{array}
\end{equation}
\label{prop_stack_convergence}
\end{proposition}

\begin{proof} In the following, the matrices $G_{yy}$ and $\tilde H_x$ are evaluated at 
$(x_t,y_t)$ and we omit the time index when no confusion arises.
\paragraph*{\color{blue}A) Out-layer Reaching (Finite-Time Entrance)}
Let
\[
A_\zeta := I-\eta_y G_{yy}, 
\quad 
A_\xi := I-\eta_x \tilde H_x .
\]
Using the spectral bounds
\[
\ell_{\zeta} I \preceq G_{yy} \preceq L_{\zeta} I,
\quad
\ell_{\xi} I \preceq \tilde H_x \preceq L_{\xi} I,
\]
the eigenvalues of $A_\zeta$ lie in 
\[
\{\,1-\eta_y\lambda:\quad\lambda\in[\ell_{\zeta},L_{\zeta}]\,\},
\]
and those of $A_\xi$ lie in
\[
\{\,1-\eta_x\lambda:\quad\lambda\in[\ell_{\xi},L_{\xi}]\,\}.
\]
Hence
\begin{align}
\|A_\zeta\| \le \alpha_\zeta := \max\{|1-\eta_y\ell_{\zeta}|,\ |1-\eta_yL_{\zeta}|\}<1,\\
\|A_\xi\| \le \alpha_\xi := \max\{|1-\eta_x\ell_{\xi}|,\ |1-\eta_xL_{\xi}|\}<1,
\end{align}
whenever
\[
0<\eta_y<\dfrac{2}{L_{\zeta}}, \quad 0<\eta_x<\dfrac{2}{L_{\xi}} .
\]
From Lemma~\ref{lemma_stack}, the recurrence for $\zeta_t$ outside the boundary  layer $\{\|\zeta_t\|\le\rho_\zeta\}$ reads (when $\|\zeta_t\|>\rho_\zeta$)
\[
\zeta_{t+1} = A_\zeta \zeta_t - \epsilon_\zeta 
\frac{\zeta_t}{\|\zeta_t\|} + \nu_t .
\]
Taking norms and using submultiplicativity,
\[
\|\zeta_{t+1}\|
\le
\|A_\zeta\|\,\|\zeta_t\|
-\epsilon_\zeta
+\|\nu_t\|
\le
\alpha_\zeta \|\zeta_t\|
-\epsilon_\zeta
+\bar\nu .
\]
Define the excess radius
\(
r_t := \|\zeta_t\|-\rho_\zeta >0 .
\)
Then
\begin{align*}
r_{t+1}
&=\|\zeta_{t+1}\|-\rho_\zeta\le
\alpha_\zeta(\rho_\zeta+r_t)-\epsilon_\zeta+\bar\nu-\rho_\zeta\\
&=
\alpha_\zeta r_t
-
(\epsilon_\zeta-\alpha_\zeta\rho_\zeta-\bar\nu).
\end{align*}
By hypothesis
\(
\epsilon_\zeta>\alpha_\zeta\rho_\zeta+\bar\nu,
\)
so the decrement term is positive. The scalar recurrence
\[
r_{t+1}\le \alpha_\zeta r_t-\Delta_\zeta
\]
with
\[
\Delta_\zeta := \epsilon_\zeta-\alpha_\zeta\rho_\zeta-\bar\nu
\]
implies finite-time entrance into the layer. Specifically,
\begin{equation}
N_\zeta
=
\Big\lceil
\frac{\|\zeta_0\|-\rho_\zeta}
{\epsilon_\zeta-\alpha_\zeta\rho_\zeta-\bar\nu}
\Big\rceil_+ .    
\end{equation}
An identical argument for $\xi_t$ yields the entrance time
\begin{equation}
    N_\xi
=
\Big\lceil
\frac{\|\xi_0\|-\rho_\xi}
{\epsilon_\xi-\alpha_\xi\rho_\xi-\bar\mu}
\Big\rceil_+ .
\end{equation}
Setting
\begin{equation}
  T_\star := \max\{N_\zeta,N_\xi\}  
\end{equation}
which gives
\[
\|\zeta_t\|\le\rho_\zeta,
\quad
\|\xi_t\|\le\rho_\xi
\]
for all $t\ge T_\star$, which proves item (i).
\paragraph*{\color{blue}B) In-layer Contraction (Chattering-Free Convergence)}
Inside the boundary layers the saturation operates in its linear regime:
\[
\mathrm{sat}(\zeta_t/\rho_\zeta)=\zeta_t/\rho_\zeta,
\quad
\mathrm{sat}(\xi_t/\rho_\xi)=\xi_t/\rho_\xi .
\]
The dynamics become linear perturbations
\begin{align}
\zeta_{t+1}&=F_\zeta\,\zeta_t+\nu_t,
\quad
F_\zeta := A_\zeta-\frac{\epsilon_\zeta}{\rho_\zeta}I,
\\
\xi_{t+1}&=F_\xi\,\xi_t+\mu_t,
\quad
F_\xi := A_\xi-\frac{\epsilon_\xi}{\rho_\xi}I .
\end{align}
By assumption
\[
\|F_\zeta\|=L_\zeta^{\mathrm{in}}<1,
\quad
\|F_\xi\|=L_\xi^{\mathrm{in}}<1.
\]
Unfolding the recurrence from $T_\star$ yields
\begin{align}
\zeta_t
&=
F_\zeta^{t-T_\star}\zeta_{T_\star}
+
\sum_{k=T_\star}^{t-1}F_\zeta^{t-1-k}\nu_k ,
\\
\xi_t
&=
F_\xi^{t-T_\star}\xi_{T_\star}
+
\sum_{k=T_\star}^{t-1}F_\xi^{t-1-k}\mu_k .
\end{align}
Taking norms and using the disturbance bounds,
\begin{align}
\|\zeta_t\|
&\le
(L_\zeta^{\mathrm{in}})^{t-T_\star}\|\zeta_{T_\star}\|
+
\frac{\bar\nu}{1-L_\zeta^{\mathrm{in}}},
\\
\|\xi_t\|
&\le
(L_\xi^{\mathrm{in}})^{t-T_\star}\|\xi_{T_\star}\|
+
\frac{\bar\mu}{1-L_\xi^{\mathrm{in}}}.
\end{align}
Taking $\limsup_{t\to\infty}$ gives
\begin{align}
\limsup_{t\to\infty}\|\zeta_t\|
&\le
\max\!\Big\{\rho_\zeta,\tfrac{\bar\nu}{1-L_\zeta^{\mathrm{in}}}\Big\},
\\
\limsup_{t\to\infty}\|\xi_t\|
&\le
\max\!\Big\{\rho_\xi,\tfrac{\bar\mu}{1-L_\xi^{\mathrm{in}}}\Big\}.
\end{align}
Thus the sliding layers are invariant and the dynamics contract inside them, 
which proves item (ii).
\end{proof}

\section{Validation and Numerical Experiments}\label{sec5}
In this section, we numerically validate the proposed SSMA and investigate its convergence, accuracy, and computational efficiency. The experiments are designed to compare SSMA with both first-order and second-order methods, and in
particular to examine whether the proposed use of curvature information provides an advantage in terms of convergence speed and accuracy. To this end, we consider the nonlinear Stackelberg problem defined by the leader objective
\begin{equation}
f(x,y)=(x-1)^2+0.60xy+\frac{0.30}{2}y^2,
\end{equation}
and the follower objective
\begin{equation}
g(x,y)=(y-h(x))^2+0.10y^4,~h(x)=0.80x-0.20x^3.
\end{equation}
We evaluate the proposed Algorithm SSMA against a broad set of first-order and curvature-aware bilevel optimization methods. The comparison includes nine algorithms: SSMA (proposed), GDA, OGDA, BR-GD, FR, HessianFR, IFT, IFT-Hessian method, and CRN bilevel method.  It is worth noting that the distinction between FR and HessianFR, and between IFT and IFT-Hessian, is not simply the presence or absence of second-order information. FR and IFT already exploit follower curvature through the local best-response geometry, e.g., $r'(x)=-g_{yy}^{-1}g_{yx}$. Their Hessian variants additionally exploit the curvature of the reduced leader objective to construct a Newton/Hessian-type leader update. Thus, FR and IFT are curvature-aware methods, whereas HessianFR and IFT-Hessian employ explicit second-order leader updates.

For a fair comparison, all methods use the same problem, initial
conditions, iteration horizon $T=1000$, numerical safeguards, and
common evaluation metric. The corresponding leader and follower gains are set to $0.0135$. All methods are evaluated using the joint Stackelberg residual
\begin{equation}
R_k=\sqrt{\zeta_k^2+\xi_k^2},
\end{equation}
with success uniformly defined as $R_T<10^{-5}.$

The deterministic experiment starts from $(x_0,y_0)=(-1.5,1.5)$.
Robustness is evaluated using $500$ Monte Carlo trials with independently sampled initial conditions
\begin{equation}\label{a}
x_0,y_0\sim\mathcal{U}[-1.5,1.5],
\end{equation}
where the same sampled initial condition is used for every algorithm in each trial. Performance is evaluated in terms of final residual, success rate, convergence to the prescribed tolerance, and computation time. In particular, speed is measured by the actual computation time required to first reach $R_k<10^{-5}$. Comparison against both first- and second-order baselines therefore evaluates the accuracy and computational efficiency of SSMA while isolating the benefit of its particular use of curvature information.

All simulations were performed in MATLAB R2024b (Update 9) on a
Windows 11 laptop equipped with an Intel Core Ultra 9 185H processor and 32~GB of RAM. Computation times were measured using MATLAB's \texttt{tic}/\texttt{toc} timing functions under the same computational environment for all methods. The time-to-tolerance metric reports the measured computation time required to first attain $R_k<10^{-5}$.

\begin{figure}[t]
    \centering

    \begin{subfigure}[t]{0.5\columnwidth}
        \centering
        \includegraphics[
            width=\linewidth,
            trim=8 5 8 5,
            clip
        ]{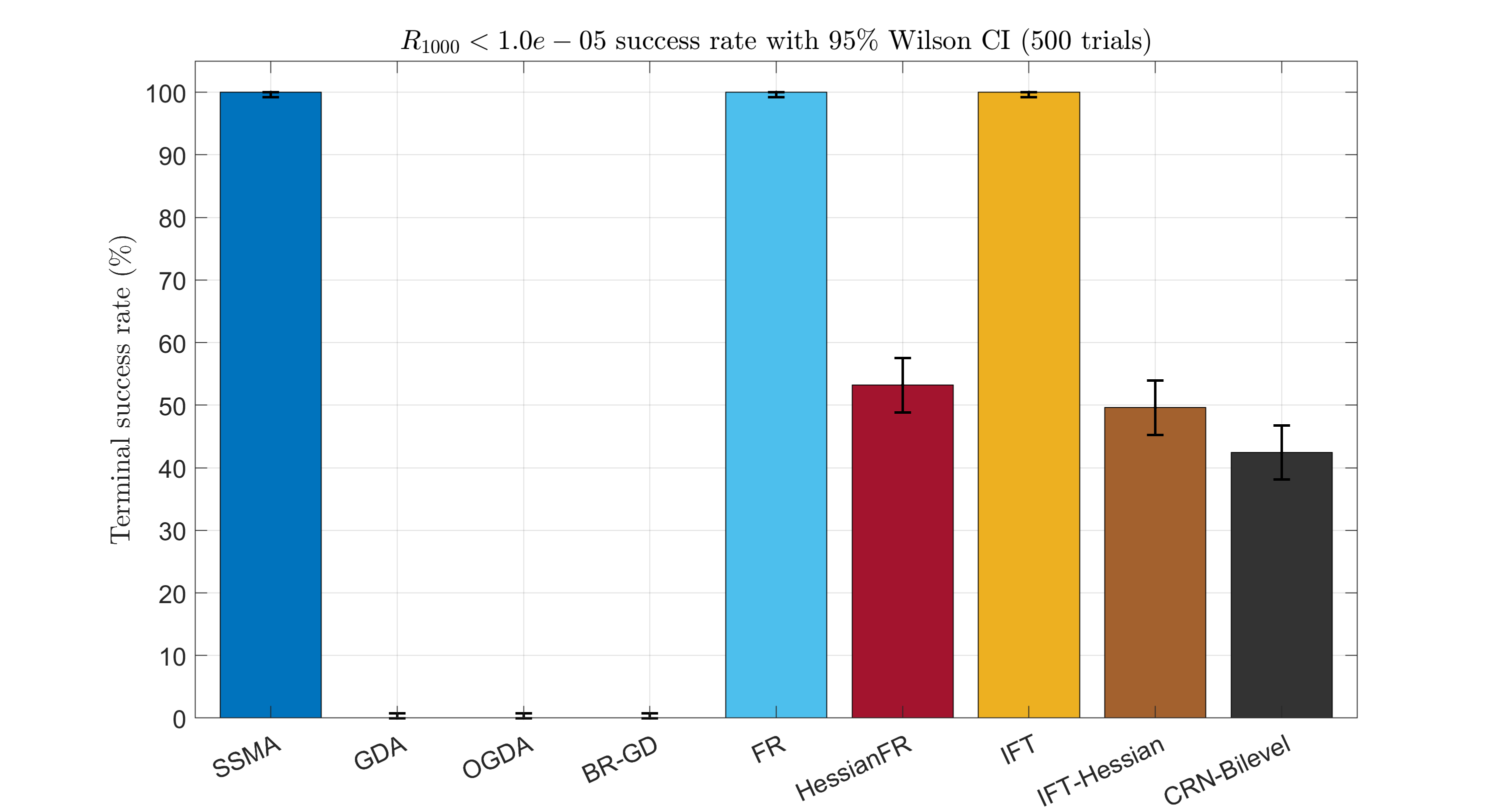}
        \caption{}
        \label{fig:mc_success}
    \end{subfigure}%
    \begin{subfigure}[t]{0.5\columnwidth}
        \centering
        \includegraphics[
            width=\linewidth,
            trim=8 5 8 5,
            clip
        ]{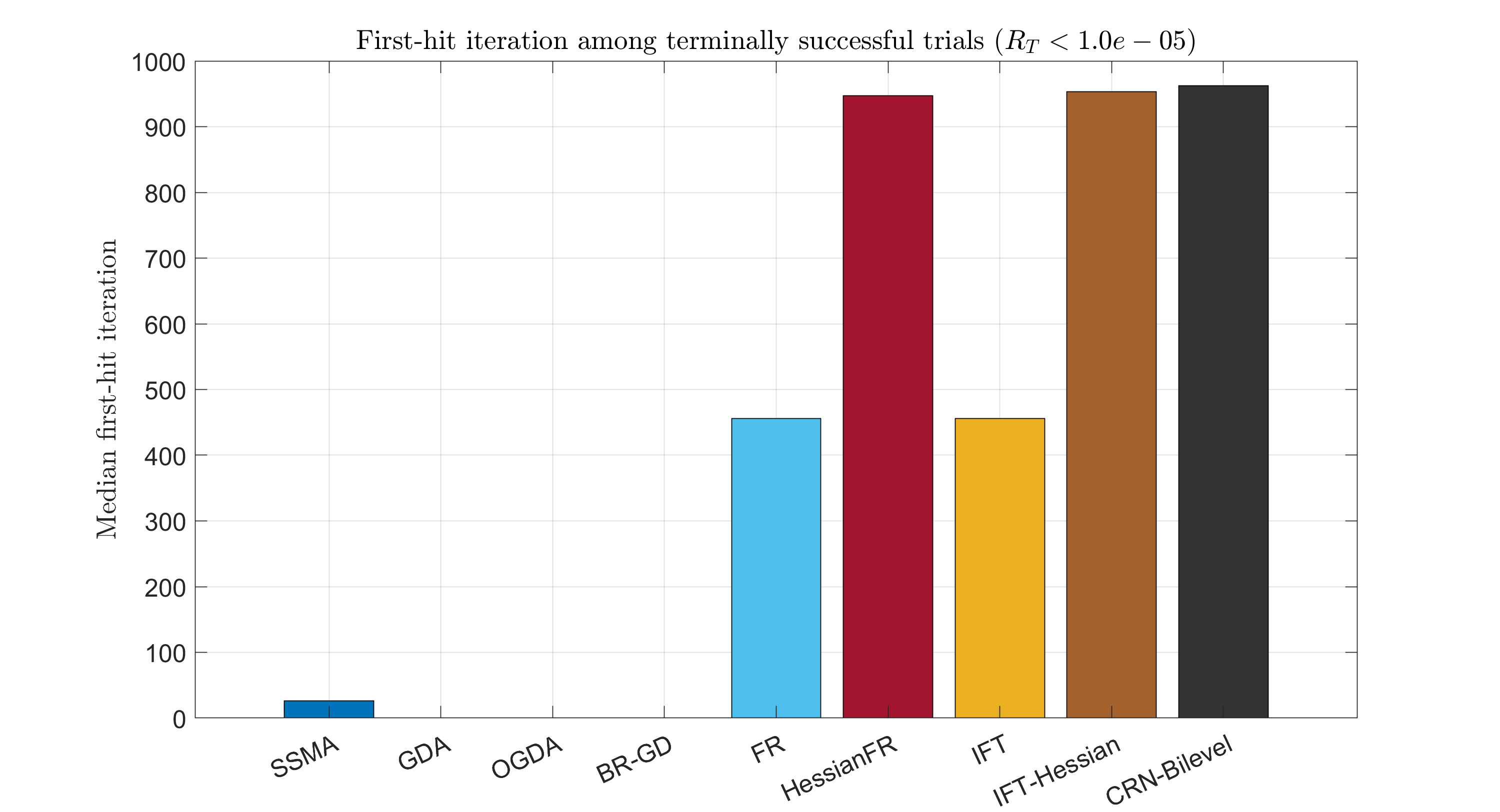}
        \caption{}
        \label{fig:mc_firsthit}
    \end{subfigure}

    \vspace{-3mm}

    \begin{subfigure}[t]{0.5\columnwidth}
        \centering
        \includegraphics[
            width=\linewidth,
            trim=8 5 8 5,
            clip
        ]{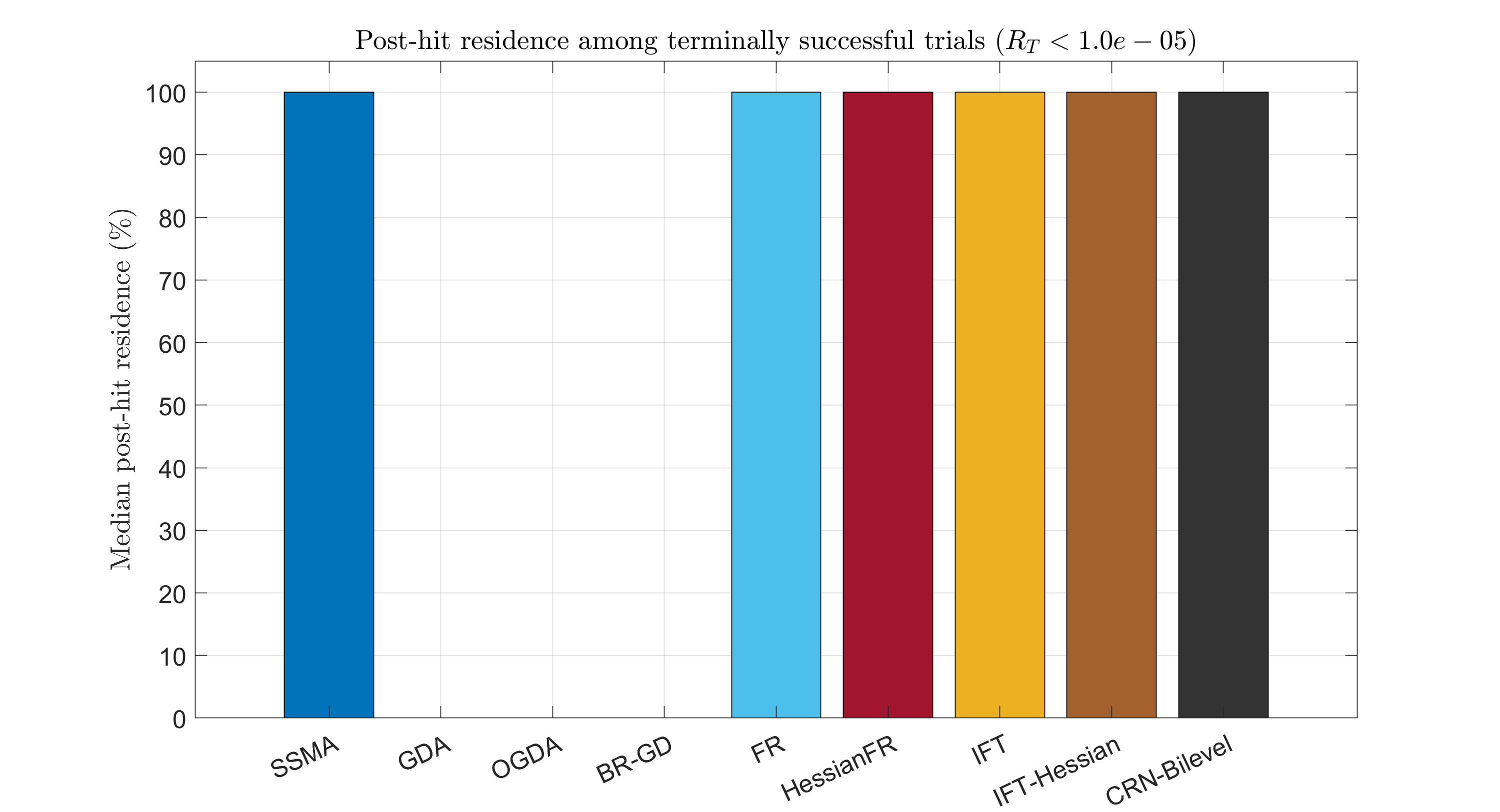}
        \caption{}
        \label{fig:mc_residence}
    \end{subfigure}%
    \begin{subfigure}[t]{0.5\columnwidth}
        \centering
        \includegraphics[
            width=\linewidth,
            trim=8 5 8 5,
            clip
        ]{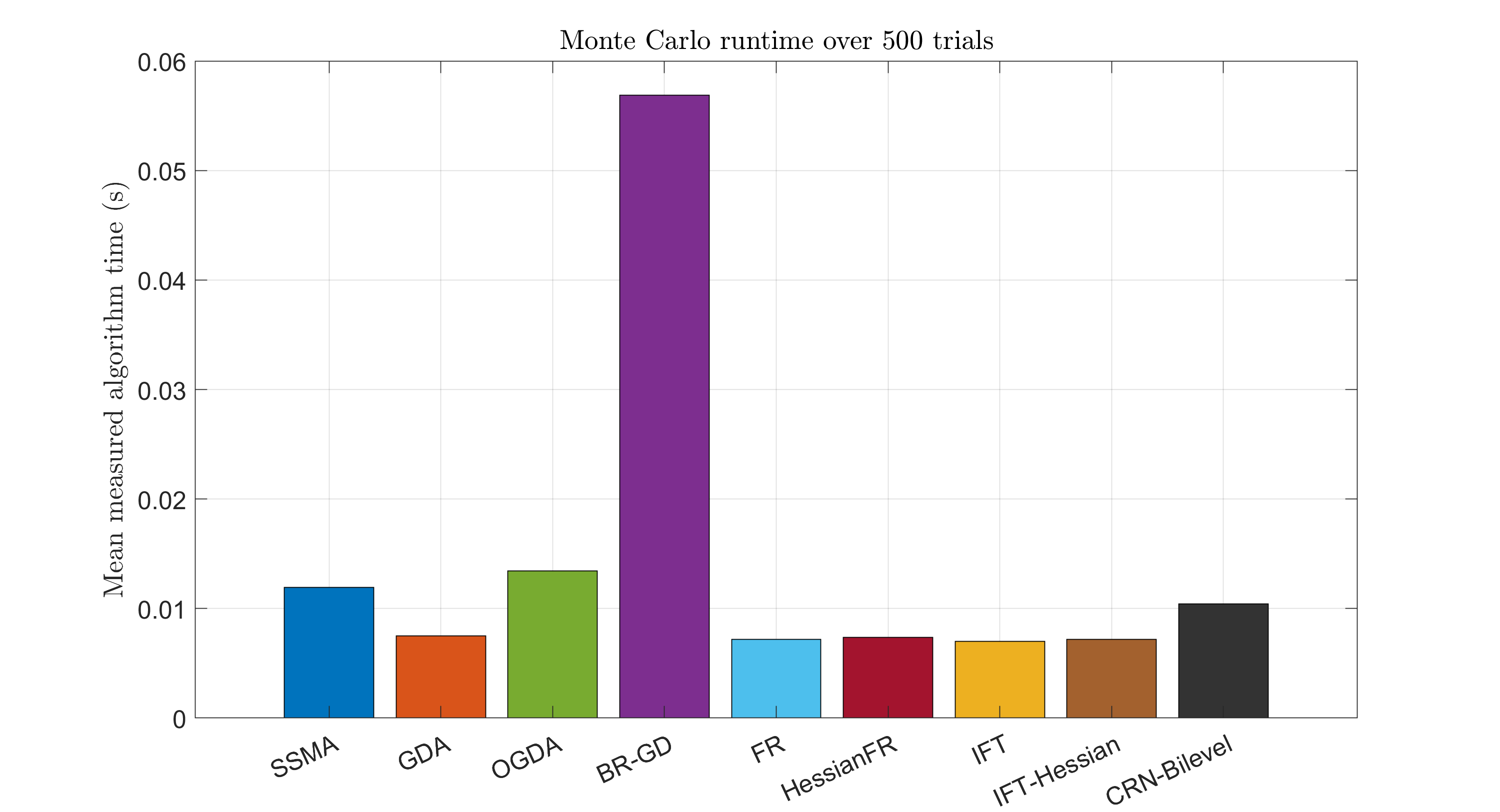}
        \caption{}
        \label{fig:mc_runtime}
    \end{subfigure}

    \vspace{-2mm}

    \caption{Monte Carlo comparison over $500$ trials, with
    $(x_0,y_0)\sim\mathcal{U}(^2)$ and $T=1000$.
    SSMA combines a $100\%$ terminal success rate with the earliest
    median attainment of the prescribed residual tolerance.
    (a) Terminal success rate with $95\%$ Wilson confidence intervals.
    (b) Median first-hit iteration among terminally successful trials.
    (c) Median post-hit residence within $R_k<10^{-5}$.
    (d) Mean measured runtime per trial.}
    \label{fig:mc_summary}
\end{figure}

\begin{figure}[thpb]
    \centering
    \includegraphics[width=\columnwidth]{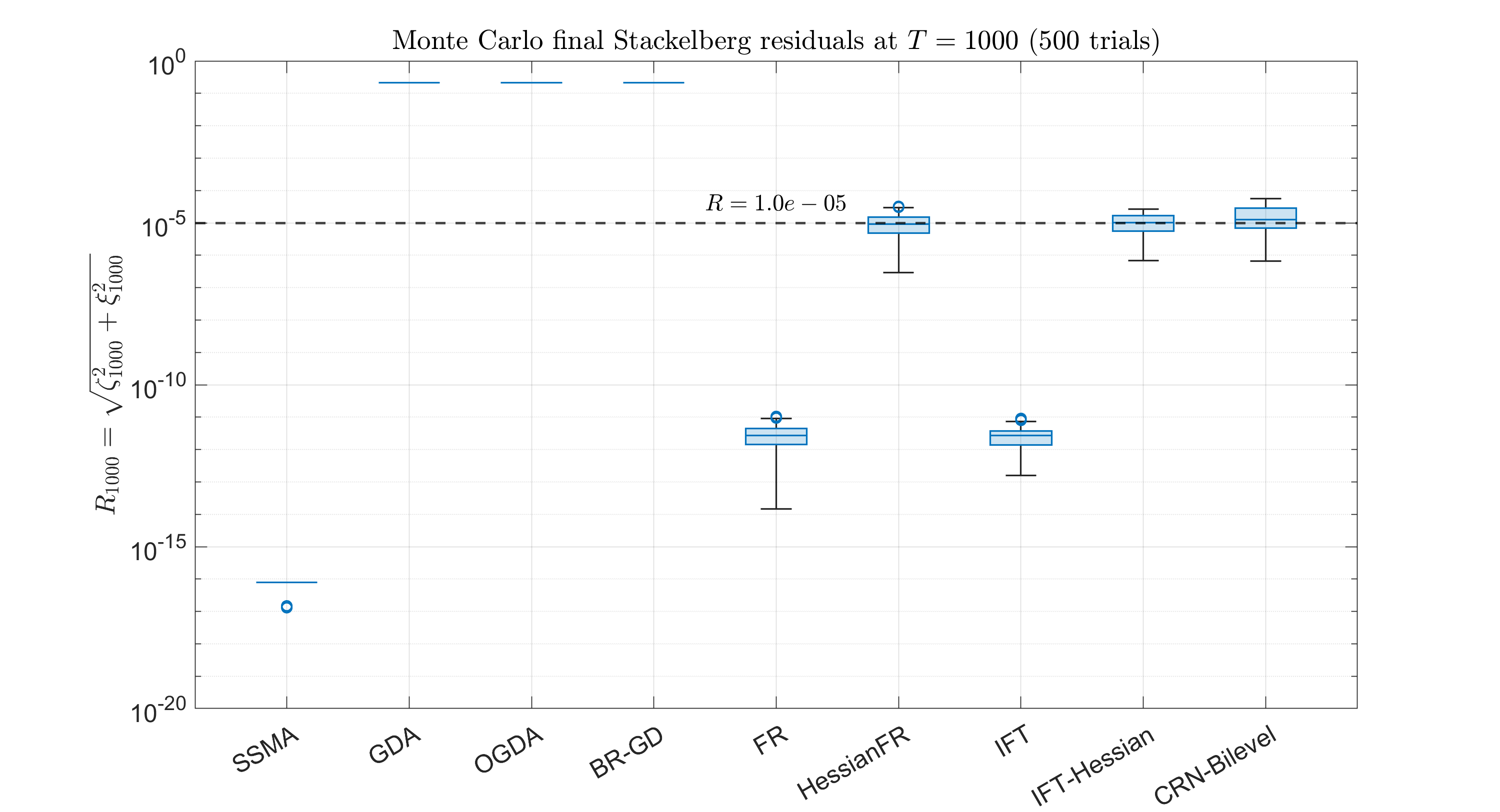}
\caption{Distribution of the terminal Stackelberg residual
$R_{1000}=\sqrt{\zeta_{1000}^{2}+\xi_{1000}^{2}}$ over $500$
Monte Carlo trials, with $(x_0,y_0)\sim\mathcal{U}(^2)$
and $T=1000$. SSMA consistently achieves residuals near numerical
precision, substantially below those of the comparison methods. FR and
IFT also converge below the prescribed tolerance but attain higher
terminal residuals, whereas HessianFR, IFT-Hessian, and CRN-Bilevel remain concentrated near the success boundary. GDA, OGDA, and BR-GD remain well above the required tolerance. The dashed line denotes the terminal success threshold $R_{1000}=10^{-5}$.}
    \label{fig:mc_final_residual}
\end{figure}

\begin{figure}[thpb]
    \centering
    \includegraphics[width=\columnwidth]
    {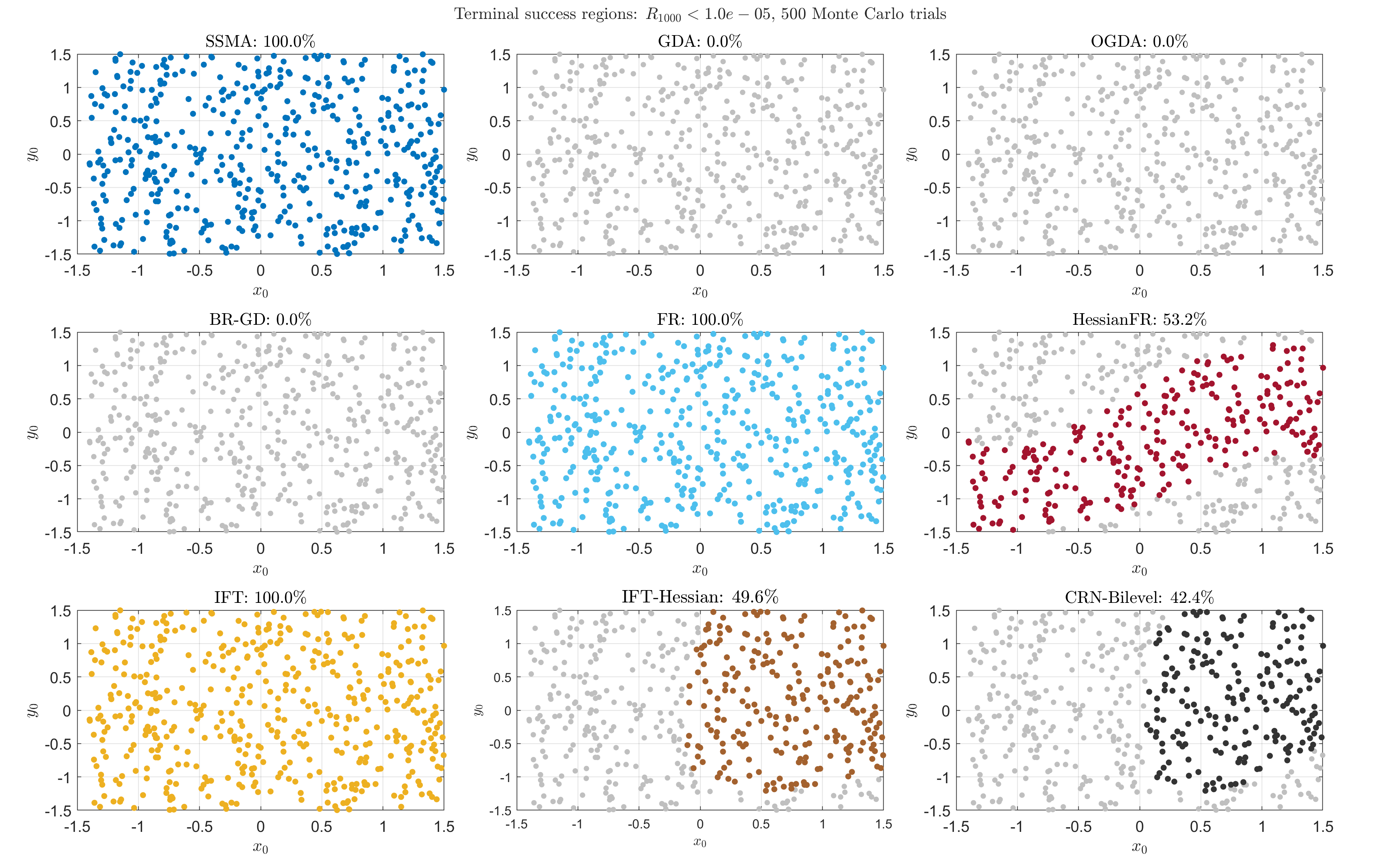}
    \caption{Terminal success regions over $500$ Monte Carlo trials.
    Coloured markers indicate initial conditions satisfying
    $R_{1000}<10^{-5}$, while grey markers indicate unsuccessful trials.}
    \label{fig:mc_success_regions}
\end{figure}
\begin{figure}[thpb]
    \centering
    \includegraphics[width=\columnwidth]{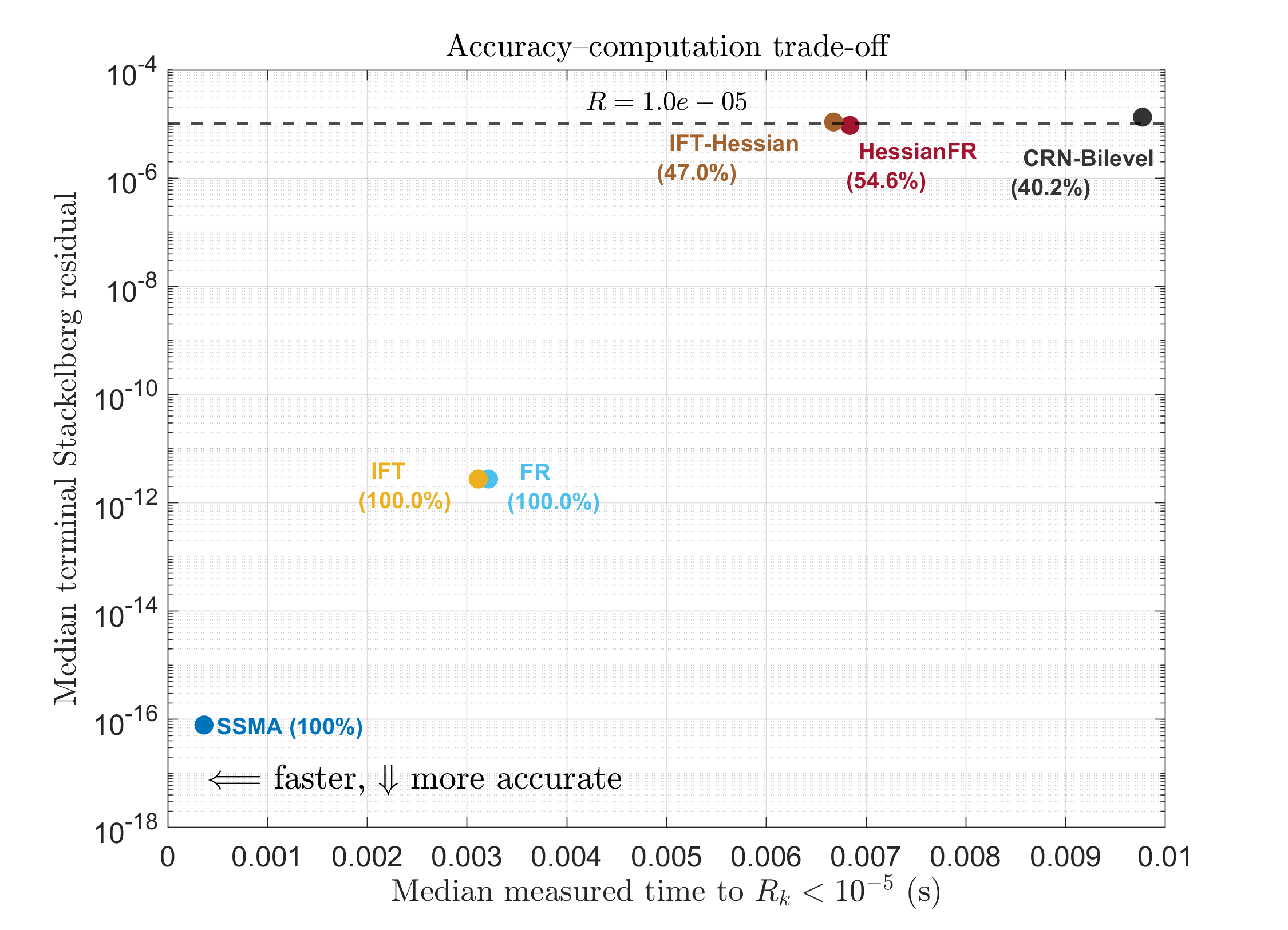}
    \caption{Accuracy--computation trade-off over $500$ Monte Carlo trials,
with $(x_0,y_0)\sim\mathcal{U}([-1.5,1.5]^2)$ and $T=1000$.
The horizontal axis reports the median measured computation time to
first satisfy $R_k<10^{-5}$, computed over successful trials, while
the vertical axis reports the median terminal Stackelberg residual.
Percentages indicate terminal success rates. Lower and farther left
correspond to greater terminal accuracy and faster attainment of the
prescribed tolerance, respectively. SSMA occupies the preferred
lower-left region, attaining near-machine-precision residuals with the shortest measured time to tolerance and a $100\%$ terminal success
rate. Methods that never reach the prescribed tolerance are omitted.}
    \label{fig:accuracy_time}
\end{figure}

Figure~\ref{fig:mc_summary} summarizes the Monte Carlo results over $500$ trials. As shown in Fig.~\ref{fig:mc_summary}-(a), SSMA, FR, and IFT attain a $100\%$ terminal success rate, where success is defined by $R_{1000}<10^{-5}$. HessianFR, IFT-Hessian, and CRN-Bilevel attain success rates of $53.2\%$, $49.6\%$, and $42.4\%$, respectively, while GDA, OGDA, and BR-GD do not satisfy the terminal criterion. The $95\%$ Wilson confidence interval for a $100\%$ success rate over $500$ trials is $[99.2,100]\%$, demonstrating the reliability of SSMA across the sampled initial conditions.

Figure~\ref{fig:mc_summary}-(b) highlights the substantially faster attainment of the prescribed tolerance by SSMA. It reaches
$R_k<10^{-5}$ in a median of only $26$ iterations, approximately
$17.5$ times earlier than FR and IFT, which both require $456$
iterations. HessianFR, IFT-Hessian, and CRN-Bilevel require median
first-hit iterations of $947.5$, $953$, and $962$, respectively,
more than $36$ times that of SSMA. GDA, OGDA, and BR-GD have no
terminally successful trials, and hence their conditional median
first-hit iterations are undefined.

Figure~\ref{fig:mc_summary}-(c) further demonstrates the persistence of convergence. The median post-hit residence is $100\%$ for every method with terminally successful trials, indicating that, for at least half of their successful trials, the residual remains below the prescribed tolerance after its first attainment. Finally, Fig.~\ref{fig:mc_summary}-(d) compares the mean measured runtimes. BR-GD incurs the largest runtime, whereas the remaining methods are substantially faster. Although SSMA does not achieve the lowest raw runtime, it provides the strongest combination of convergence reliability, rapid attainment of the prescribed tolerance, and persistent convergence. Thus, the principal result of Fig.~\ref{fig:mc_summary} is the fast, reliable, and persistent
convergence exhibited by SSMA.

Figure~\ref{fig:mc_final_residual} further reveals a clear separation in terminal equilibrium accuracy. SSMA produces the lowest and most consistent terminal Stackelberg residuals across the $500$ trials, with a median of $7.85\times10^{-17}$, near numerical precision. FR and IFT attain median residuals of $2.73\times10^{-12}$ and $2.70\times10^{-12}$, respectively, approximately four to five orders of magnitude larger than that of SSMA, although both remain well below the prescribed $10^{-5}$ tolerance. In contrast, HessianFR, IFT-Hessian, and CRN-Bilevel remain concentrated near the success threshold, while GDA, OGDA, and BR-GD remain well above it. These results demonstrate that SSMA combines reliable convergence with substantially higher terminal equilibrium accuracy.

Taken together, Figs.~\ref{fig:mc_summary} and
\ref{fig:mc_final_residual} show that SSMA reaches the prescribed
equilibrium tolerance substantially earlier than the competing
successful methods, attains the lowest terminal Stackelberg residual, and maintains a $100\%$ terminal success rate across the sampled initial conditions. Thus, the observed advantage of SSMA is not confined to a single performance metric but reflects the combined improvement in convergence speed, terminal accuracy, and reliability.

The initial-condition maps in Fig.~\ref{fig:mc_success_regions} provide a spatial view of this robustness. SSMA, FR, and IFT succeed throughout the sampled domain. In contrast, HessianFR, IFT-Hessian, and CRN-Bilevel exhibit restricted success regions, showing greater sensitivity to initialisation under the same numerical settings. Finally, iteration count alone does not account for differences in per-iteration computational cost. We therefore measure the actual algorithm computation time required to first satisfy $R_k<10^{-5}$.

Figure~\ref{fig:accuracy_time} summarises the combined accuracy,
computational speed, and reliability of the successful methods. The time-to-tolerance statistic is computed only over trials that
successfully attain $R_k<10^{-5}$ within the $T=1000$ iteration
horizon. SSMA occupies the preferred lower-left region, attaining both the smallest median terminal residual and the shortest measured time to the prescribed tolerance while maintaining a $100\%$ terminal success rate. Compared with FR and IFT, which also achieve $100\%$ terminal success, SSMA reaches the tolerance substantially faster and reduces the median terminal residual by approximately four to five orders of magnitude. HessianFR, IFT-Hessian, and CRN-Bilevel require longer times to reach the tolerance and exhibit substantially lower terminal success rates. GDA, OGDA, and BR-GD are omitted because none of their trials reaches the prescribed tolerance, making their time-to-tolerance
undefined. Thus, the advantage of SSMA is not confined to a single
metric, but is reflected jointly in convergence speed, terminal
accuracy, and reliability.

Table~\ref{tab:mc_results} consolidates the Monte Carlo results across the principal performance measures. SSMA attains a $100\%$ terminal success rate and reaches the prescribed tolerance in a median of only $26$ iterations, compared with $456$ iterations for both FR and IFT. This corresponds to an approximately $94\%$ reduction in the median number of iterations required to attain the tolerance. SSMA also achieves the smallest median terminal residual, $7.85\times10^{-17}$, compared with approximately
$2.7\times10^{-12}$ for FR and IFT. In contrast, HessianFR,
IFT-Hessian, and CRN-Bilevel require nearly the full iteration horizon to reach the tolerance on their successful trials and attain terminal success rates of only $53.2\%$, $49.6\%$, and $42.4\%$, respectively. GDA, OGDA, and BR-GD do not satisfy the terminal criterion in any of the $500$ trials.

\begin{table}[thpb]
\centering
\caption{\textit{Monte Carlo comparison over $500$ trials, with
$(x_0,y_0)\sim\mathcal{U}([-1.5,1.5]^2)$ and $T=1000$.
Here, $P_s$ denotes the terminal success rate, defined by
$R_T<10^{-5}$, and $\mathrm{CI}_{95}$ denotes the corresponding
$95\%$ Wilson confidence interval. The statistics
$\widetilde{R}_T$ and $\overline{R}_T$ denote, respectively, the
median and mean terminal Stackelberg residuals computed over all
$500$ trials, whereas $\widetilde{k}_{\mathrm{hit}}$ denotes the
median first-hit iteration, computed only over trials that attain
$R_k<10^{-5}$ within the prescribed horizon. The LSE column indicates whether the attained solution, upon convergence, satisfies the second-order conditions.}}
\label{tab:mc_results}

\scriptsize
\setlength{\tabcolsep}{2pt}
\renewcommand{\arraystretch}{1.05}

\resizebox{\columnwidth}{!}{%
\begin{tabular}{lcccccc}
\hline
\textbf{Method} &
$\mathbf{P_s}$ (\%) &
$\mathbf{CI_{95}}$ &
$\mathbf{\widetilde{R}_T}$ &
$\mathbf{\overline{R}_T}$ &
$\mathbf{\widetilde{k}_{\mathrm{hit}}}$ &
\textbf{LSE} \\
\hline

SSMA
& 100
& [99.2,100]
& $7.85\times10^{-17}$
& $7.71\times10^{-17}$
& 26
& \cmark \\

GDA
& 0.0
& [0.0,0.8]
& $2.13\times10^{-1}$
& $2.13\times10^{-1}$
& --
& \xmark \\

OGDA
& 0.0
& [0.0,0.8]
& $2.13\times10^{-1}$
& $2.13\times10^{-1}$
& --
& \xmark \\

BR-GD
& 0.0
& [0.0,0.8]
& $2.13\times10^{-1}$
& $2.13\times10^{-1}$
& --
& \xmark \\

FR
& 100
& [99.2,100]
& $2.73\times10^{-12}$
& $3.12\times10^{-12}$
& 456
& \cmark \\

HessianFR
& 53.2
& [48.8,57.5]
& $9.30\times10^{-6}$
& $1.06\times10^{-5}$
& 947.5
& \cmark \\

IFT
& 100
& [99.2,100]
& $2.70\times10^{-12}$
& $2.74\times10^{-12}$
& 456
& \cmark \\

IFT-Hessian
& 49.6
& [45.2,54.0]
& $1.00\times10^{-5}$
& $1.11\times10^{-5}$
& 953
& \cmark \\

CRN-Bilevel
& 42.4
& [38.1,46.8]
& $1.25\times10^{-5}$
& $1.80\times10^{-5}$
& 962
& \cmark \\

\hline
\end{tabular}%
}
\end{table}

Overall, the Monte Carlo experiments demonstrate that SSMA combines
robustness over the sampled initial-condition domain, rapid attainment
of the prescribed Stackelberg tolerance, and high terminal accuracy.
Across $500$ trials, SSMA achieves a $100\%$ terminal success rate,
reaches the prescribed tolerance in a median of only $26$ iterations,
and attains a median terminal residual of $7.85\times10^{-17}$.
The comparison with both gradient-based and curvature-aware baselines
further indicates that these improvements are associated with the
proposed curvature-informed manifold-stabilization mechanism rather
than with access to second-order information alone.

For additional insight into the transient behavior, a representative
single-scenario simulation is visualized in the
\href{https://youtu.be/vToUQnh0WMc}{supplementary video}, showing the
evolution of SSMA and the comparison methods toward the Stackelberg
equilibrium.

\section{Conclusion}\label{sec6}

This paper introduced the Stackelberg Sliding Mode Algorithm (SSMA), a
control-theoretic approach that reformulates the computation of local
Stackelberg equilibria in nonlinear leader--follower games as a
manifold-stabilization problem. SSMA constructs sliding manifolds from
the follower and leader first-order optimality conditions and employs a
discrete-time reaching law to drive the iterates toward these manifolds.
The follower-curvature and leader reduced-curvature conditions determine
manifold attractivity, enabling valid local Stackelberg equilibria to be
selected while stationary points that fail the second-order equilibrium
conditions are rejected. The theoretical analysis establishes
finite-time reaching of the Stackelberg sliding manifolds under the stated assumptions. In $500$ Monte Carlo trials, SSMA was the only evaluated method to simultaneously achieve $100\%$ success, near-machine-precision terminal residuals, rapid convergence, and satisfaction of the second-order conditions for a valid local Stackelberg equilibrium. It reached the prescribed tolerance approximately $17.5$ times earlier than the other fully successful methods while reducing the terminal residual by about four orders of magnitude.

\bibliographystyle{IEEEtran}
\bibliography{ref.bib}

\addtolength{\textheight}{-12cm}   




\end{document}